\documentclass[reqno,10pt]{amsart}
\usepackage{amsmath,amssymb,latexsym,soul,cite,mathrsfs}
\usepackage{color,graphicx}
\usepackage[colorlinks=true,urlcolor=blue,
citecolor=red,linkcolor=blue,linktocpage,pdfpagelabels,
bookmarksnumbered,bookmarksopen]{hyperref}
\usepackage[english]{babel}
\usepackage[left=2.7cm,right=2.7cm,top=3.2cm,bottom=3.2cm]{geometry}
\usepackage{color,enumitem,graphicx}
\usepackage[norefs,nocites]{refcheck}

\usepackage[hyperpageref]{backref}
\numberwithin{equation}{section}

\newtheorem{theorem}{Theorem}[section]
\newtheorem*{trm}{Theorem A}
\newtheorem{lemma}{Lemma}[section]

\newtheorem{proposition}{Proposition}[section]

\newtheorem{remark}{Remark}[section]

\title[Fractional Choquard--type equation in $\mathbb{R}$ with critical exponential growth]{Existence of solutions for a fractional Choquard--type equation in $\mathbb{R}$ with critical exponential growth}

\author{Rodrigo Clemente}
\address{Department of Mathematics, Rural Federal University of Pernambuco
	\newline\indent 
	52171-900, Recife-PE, Brazil}
\email{\href{mailto:rodrigo.clemente@ufrpe.br}{rodrigo.clemente@ufrpe.br}}

\author{Jos\'e Carlos de Albuquerque}
\address{Department of Mathematics, Federal University of Pernambuco
	\newline\indent
	50670-901, Recife-PE, Brazil}
\email{\href{mailto:josecarlos.melojunior@ufpe.br ; joserre@gmail.com}{josecarlos.melojunior@ufpe.br ; joserre@gmail.com}}

\author{Eudes Barboza}
\address{Department of Mathematics, Rural Federal University of Pernambuco
	\newline\indent 
	52171-900, Recife-PE, Brazil}
\email{\href{mailto:eudes.barboza@ufrpe.br ; eudesmendesbarboza@gmail.com}{eudes.barboza@ufrpe.br ; eudesmendesbarboza@gmail.com}}



\begin{document}

\begin{abstract}
In this paper we study the following class of fractional Choquard--type equations
\[
(-\Delta)^{1/2}u + u=\Big( I_\mu \ast F(u)\Big)f(u), \quad x\in\mathbb{R},
\]
where $(-\Delta)^{1/2}$ denotes the $1/2$--Laplacian operator, $I_{\mu}$ is the Riesz potential with $0<\mu<1$ and $F$ is the primitive function of $f$. We use Variational Methods and minimax estimates to study the existence of solutions when $f$ has critical exponential growth in the sense of Trudinger--Moser inequality.
\end{abstract}
\maketitle

\section{Introduction}	

In this paper, we are concerned with existence of solutions for a class of fractional Choquard--type equations
\begin{equation}\label{i1}
(-\Delta)^{s}u + u=\Big( I_\mu \ast F(u)\Big)f(u), \quad x\in\mathbb{R}^{N},
\end{equation}
where $(-\Delta)^{s}$ denotes the fractional Laplacian, $0<s<1$, $0<\mu<N$, $F$ is the primitive function of $f$, $I_{\mu}:\mathbb{R}^{N}\backslash\{0\}\rightarrow\mathbb{R}$ is the Riesz potential defined by
\[
I_{\mu}(x):=\mathcal{A}_{\mu}\frac{1}{|x|^{N-\mu}}, \quad \mbox{where} \quad \mathcal{A}_{\mu}:=\dfrac{\displaystyle\Gamma\left(\frac{N-\mu}{2} \right)}{\displaystyle\Gamma\left(\frac{\mu}{2}\right)\pi^{\frac{N}{2}}2^{\mu}},
\]
and $\Gamma$ denotes the Gamma function. We consider the ``limit case" when $N=1$, $s=1/2$ and a Choquard--type nonlinearity with critical exponential growth motivated by a class of Trudinger--Moser inequality, see \cite{Ozawa,iulamar,Wadade,taka}. The main difficulty is to overcome the ``lack of compactness" inherent to problems defined on unbounded domains or involving nonlinearities with critical growth. In order to apply properly the Variational Methods, we control the minimax level with fine estimates involving Moser functions (see \cite{taka}), but here in the context of fractional Choquard--type equation. Before stating our assumptions and main result, we introduce a brief survey on related results to motivate our problem.\\

\noindent {\bf Motivation.} Nonlinear elliptic equations involving nonlocal operators have been widely studied both from a pure mathematical point of view and their concrete applications, since they naturally arise in many different contexts, such as, among the others, obstacle problems, flame propagation, minimal surfaces, conservation laws, financial market, optimization, crystal dislocation, phase transition and water waves, see for instance \cite{caffa,guia} and references therein. The class of equations \eqref{i1} is motivated by the search of standing wave solutions for the following class of time--dependent fractional Schr\"{o}dinger equations
\begin{equation}\label{1}
i\frac{\partial\Psi}{\partial t}=(-\Delta)^{s}\Psi+W(x)\Psi-\left(I_{\mu}\ast \vert \Psi \vert^{p} \right)\vert \Psi \vert^{p-2}\Psi, \quad (t,x)\in \mathbb{R}_{+}\times\mathbb{R}^{N},
\end{equation}
where $i$ denotes the imaginary unit, $p\geq 2$, $s\in(0,1)$ and $W(x)$ is a external potential. A standing wave solution of \eqref{1} is a solution of type $\Psi(x,t)=u(x)e^{-i\omega t}$, where $\omega\in\mathbb{R}$ and $u$ solves the stationary equation
\begin{equation}\label{4}
(-\Delta)^s u +V(x)u=\left( I_{\mu}\ast \vert u\vert^{p}\right)\vert u\vert^{p-2}u,  \quad  \text{in } \mathbb{R}^N,
\end{equation}
with $V(x)=W(x)-\omega$. In some particular cases, equation \eqref{4} is also known as the Schr\"{o}dinger-Newton equation which was introduced by R. Penrose in \cite{penrose} to investigate the self-gravitational collapse of a quantum mechanical wave function. It is well known that when $s\rightarrow1$, the fractional Laplacian $(-\Delta)^{s}$ reduces to the standard Laplacian $-\Delta$, see \cite{guia}. In the local case when $s=1$, $p>1$, $\mu=2$ and $N=3$, equation \eqref{i1} becomes the following nonlinear Choquard equation
\begin{equation}\label{choqu}
-\Delta u+u=\left(I_{2}\ast |u|^{p}\right)|u|^{p-2}u, \quad x\in\mathbb{R}^{3}.
\end{equation}
This case goes back to 1954, in the work \cite{pekar}, when S. Pekar described a polaron at rest in the quantum theory. In 1976, to model an electron trapped in its own hole, P. Choquard considered equation \eqref{choqu} as a certain approximation to Hartree--Fock theory of one-component plasma, see \cite{choquard}. For more information on physical background, we refer the readers to \cite{choq1,choq2}. 

There is a large bibliography regarding to Choquard--type equations in the case of the standard Laplacian operator. In this direction we refer the readers to the seminal works \cite{choquard,lions,moroz,moroz2} and references therein.

For dimension $N = 2$, the Trudinger--Moser inequality may be viewed as a substitute of the Sobolev inequality as it establishes a maximum growth for integrability of functions on $H^1(\mathbb{R}^2)$, see \cite{Miyagaki,ad}. The first version of the Trudinger--Moser inequality in $\mathbb{R}^2$ was established by D. Cao in \cite{cao} and this fact has inspired many works for elliptic equations including Choquard--type nonlinearities, see \cite{alvmimbo,alvesetal,Mimbo,subcritical} and references therein.\\

\noindent {\bf Assumptions and main Theorem.}  Inspired by \cite{alvesetal}, our goal is to establish a link between Choquard--type equations, $1/2$-- fractional Laplacian and nonlinearity with critical exponential growth. We are interested in the following class of problems
\begin{equation}\label{problem0}\tag{$\mathcal{P}$}
(-\Delta)^{1/2}u + u=\Big( I_\mu \ast F(u)\Big)f(u), \quad x\in\mathbb{R},
\end{equation}
where $F$ is the primitive of $f$. For simplicity, we consider $I_{\mu}=\vert x\vert^{-\mu}$. In order to use a variational approach, the maximal growth is motivated by the Trudinger--Moser inequality first given by T. Ozawa \cite{Ozawa} and later extended by S. Iula, A. Maalaoui, L. Martinazzi \cite{iulamar} (see also \cite{Wadade}). Precisely, it holds
\[
\sup_{\substack{u \in H^{1/2}(\mathbb{R})\\ \|(-\Delta)^{1/4}u\|_{2} \leq
1}}\int_\mathbb{R}(e^{\alpha u^{2}}-1) \, \mathrm{d}x \leq  	
\left\{\begin{array}{rl}
< \infty,&\alpha\leq \pi,\\
=\infty, &\alpha> \pi .
\end{array}\right.
\]
In this work we suppose that $f:\mathbb{R}\rightarrow\mathbb{R}$ is a continuous function satisfying the following hypotheses: 
\begin{enumerate}[label=$(f_\arabic*)$] 
\item\label{f1} $f(t)=0$, for all $t\leq0$ and $0\leq f(t)\leq Ce^{\pi t^{2}}$, for all $t\geq0$;
\item\label{f2} There exist $t_0$, $C_0>0$ and $a\in(0,1]$ such that $0<t^aF(t)\leq C_0f(t)$, for all $t\geq t_0$;
\item\label{f3} There exist $p>1-\mu$ and $C_p=C(p)>0$ such that $f(t)\thicksim C_pt^p$, as $t\rightarrow 0$;
\item\label{f4} There exists $K>1$ such that $KF(t)<f(t)t$ for all $t>0$, where $F(t)=\int_{0}^{t}f(\tau)\,\mathrm{d}\tau$;
\item\label{f5}   $\displaystyle\liminf_{t\rightarrow +\infty}\frac{F(t)}{e^{\pi t^2}}= \sqrt{\beta_0}$ with $\beta_0>0$.
\end{enumerate}

Assumption \ref{f5} plays a very important role in estimating the minimax level to recover some compactness in our approach. For this reason, we give a few comments on this hypothesis in different contexts. Since the pioneer works \cite{ad,Miyagaki}, many authors have used Moser functions to estimate the minimax level of functional associated to a problem involving a nonlinearity $f(t)$ with exponential growth. For this matter, usually one may uses the asymptotic behavior of $h(t)={f(t)t}/{e^{\alpha_0t^2}}$ at infinity, which can appear in different ways. For instance, in \cite{Miyagaki} the authors considered (among other conditions) that $\lim_{t\rightarrow \infty}h(t)= C(r)$, where $r$ is the radius of the largest open ball in the domain. In \cite{Pawan}, it was assumed that  $\lim_{t\rightarrow \infty}h(t)= \infty$ for an equation involving $1/2$-Laplacian operator. Regarding to Choquard--type equations, due to the problem nature, this type of hypothesis must be adapted, because we need to estimate some integrals where $tf(t)$ and $F(t)$  appear simultaneously. In \cite{alvesetal} it is considered that there exists a positive constant $\gamma_0>0$ large enough such that
\[\lim_{t\rightarrow \infty} \dfrac{t\tilde{f}(t)\tilde{F}(t)}{e^{8\pi t^2}}\geq\gamma_0, \]
while in \cite{Li} (see also \cite{Ub,nlaplacian}) it is assumed that there exists a positive constant $\xi_0>0$ large enough such that

\[
\displaystyle\lim_{t\rightarrow +\infty}\frac{\tilde{F}(t)}{e^{4\pi t^2}}\geq\xi_0.\]
In our case,  differently from \cite{Li}, it is not necessary assume any constraint on constant $\beta_0>0$ in assumption \ref{f5}, similarly as it occurs in \cite{Ub}.

We are in condition to state our main result:

\begin{theorem}\label{paper4A}
Suppose that $0<\mu<1$ and assumptions \ref{f1}--\ref{f5} hold. Then, Problem \eqref{problem0} has a nontrivial weak solution.
\end{theorem}

\begin{remark}
	Though there has been many works on the existence of solutions for problem \eqref{i1}, as far as we know, this is the first work considering a fractional Choquard--type equation involving $1/2$--Laplacian operator and nonlinearity with critical exponential growth. Particularly, our Theorem \ref{paper4A} is a version of Theorem 1.3 of \cite{alvesetal} for $1/2$-Laplacian operator.
\end{remark}

\begin{remark}
	Assumptions \ref{f2} and \ref{f5} imply the asymptotic behavior of $tf(t)F(t)/e^{2\pi t^2}$ at infinity. Precisely, for given $\varepsilon>0$, there exists $t_0>0$ such that 
	\begin{equation}\label{HUB}
	{tf(t)F(t)}\geq (\beta_0 -\varepsilon)C_0{e^{2\pi t^2}} t^{a+1}, \quad \mbox{for all } t>t_{0}.
	\end{equation}
	This behaviour plays a very important role to estimate the minimax level associated to Problem \eqref{problem0}, using a version of Moser functions for problems involving $1/2$--Laplacian operator.
\end{remark}

\noindent {\bf Outline.} The paper is organized as follows: In the forthcoming Section we recall some definitions and preliminary basic results which are important to prove our main result. In Section \ref{tpmm} we introduce the variational setting and we study the mountain pass geometry. Section \ref{mini} is devoted to study minimax estimates. Precisely, we establish an upper estimate of the minimax level that guarantees some compactness of Palais--Smale sequences. In Section \ref{fs} we prove Theorem \ref{paper4A}.

\section{Preliminaries}

We start this Section recalling some preliminary concepts about the fractional operator, for a more complete discussion we cite \cite{guia}. For $s\in(0,1)$, the \textit{fractional Laplacian operator} of a measurable function $u:\mathbb{R}^{N}\rightarrow\mathbb{R}$ may be defined by
$$
(-\Delta)^{s}u(x)=-\frac{C(N,s)}{2}\int_{\mathbb{R}^{N}}\frac{u(x+y)+u(x-y)-2u(x)}{|y|^{N+2s}}\;\mathrm{d} y, \quad \mbox{for all} \hspace{0,2cm} x\in\mathbb{R}^{N},
$$
for some normalizing constant $C(N,s)$. The particular case when $s=1/2$ its called the \textit{square root of the Laplacian}. We recall the definition of the fractional Sobolev space
$$
H^{1/2}(\mathbb{R})=\left\{u\in L^{2}(\mathbb{R}):\int_{\mathbb{R}}\int_{\mathbb{R}}\frac{|u(x)-u(y)|^{2}}{|x-y|^{2}}\;\mathrm{d} x\mathrm{d} y<\infty\right\},
$$
endowed with the standard norm
$$
\|u\|_{1/2}=\left(\frac{1}{2\pi}[u]_{1/2}^{2}+\int_{\mathbb{R}}u^{2}\;\mathrm{d} x\right)^{1/2},
$$
where the term 
$$
[u]_{1/2}=\left(\int_{\mathbb{R}}\int_{\mathbb{R}}\frac{|u(x)-u(y)|^{2}}{|x-y|^{2}}\;\mathrm{d} x\mathrm{d} y\right)^{1/2}
$$
is the so-called \textit{Gagliardo semi-norm} of the function $u$. We point out from \cite[Proposition 3.6]{guia} that
\begin{equation}\label{normas}
\|(-\Delta)^{1/4}u\|_{L^{2}(\mathbb{R})}^{2}=\frac{1}{2\pi}\int_{\mathbb{R}}\int_{\mathbb{R}}\frac{|u(x)-u(y)|^{2}}{|x-y|^{2}}\;\mathrm{d} x\mathrm{d} y, \quad \mbox{for all} \hspace{0,2cm} u\in H^{1/2}(\mathbb{R}).
\end{equation}

In order to deal with the exponential growth we use the following result due to S. Iula, A. Maalaoui, L. Martinazzi, see \cite[Theorem 1.5]{iulamar}:

\begin{trm}[Fractional Trudinger-Moser inequality]\label{tmm}
We have
\[
\sup_{\substack{u \in H^{1/2}(\mathbb{R})\\ \|u\|_{1/2} \leq
1}}\int_\mathbb{R}(e^{\pi u^{2}}-1) \, \mathrm{d}x <+\infty.
\]
Moreover, for any $a>2$,
\[
\sup_{\substack{u \in H^{1/2}(\mathbb{R})\\ \|u\|_{1/2} \leq
1}}\int_\mathbb{R}|u|^a(e^{\pi u^{2}}-1) \, \mathrm{d}x =+\infty.
\]
\end{trm}

The \textit{vanishing lemma} was proved originally by P.L.~Lions \cite[Lemma~I.1]{lionss} and here we use the following version to fractional Sobolev spaces:

\begin{lemma}\label{lion}
Assume that $(u_{n})$ is a bounded sequence in $H^{1/2}(\mathbb{R})$ satisfying
\begin{align*}
\lim_{n\rightarrow+\infty}\sup_{y\in\mathbb{R}}\int_{y-R}^{y+R}|u_{n}|^{2}\;\mathrm{d} x=0,
\end{align*}
for some $R>0$. Then, $u_{n}\rightarrow0$ strongly in $L^{p}(\mathbb{R})$, for $2<p<\infty$.
\end{lemma}

In order to study the convolution term we use the Hardy-Littlewood-Sobolev inequality, which can be stated as follows:
\begin{lemma}[Hardy-Littlewood-Sobolev inequality]
Let $1<r,t<\infty$ and $0<\mu<N$ with $1/r+1/t+\mu/N=2$. If $f\in L^{r}(\mathbb{R}^{N})$ and $h\in L^{t}(\mathbb{R}^{N})$, then there exists a sharp constant $C=C(r,t,\mu)>0$, independent of $f$ and $h$, such that
\[
\int_{\mathbb{R}^{N}}\int_{\mathbb{R}^{N}}\frac{f(x)h(y)}{|x-y|^{\mu}}\,\mathrm{d}x\mathrm{d}y\leq C\|f\|_{L^{r}(\mathbb{R}^{N})}\|h\|_{L^{t}(\mathbb{R}^{N})}.
\]
\end{lemma}

\section{The Variational Framework}\label{tpmm}

In this section we introduce the variational framework to our problem. The energy functional $I: H^{1/2}(\mathbb{R}) \rightarrow \mathbb{R}$ associated to Problem \eqref{problem0} is defined by
\begin{equation}\label{ef} 
I(u)=\frac{1}{2}\|u\|_{1/2}^{2}-\frac{1}{2}\int_{\mathbb{R}} (I_{\mu}\ast F(u))F(u)\, \mathrm{d} x
\end{equation}
where $\displaystyle{F(t)=\int_0^t f(\tau)\;\mathrm{d} \tau}$. By using assumptions \ref{f1} and \ref{f3}, it follows that for each  $q>2$ and $\varepsilon>0$ there exists $C_{\varepsilon}>0$ such that 
\begin{equation}\label{growth}
f(t)\leq \varepsilon |t|^{1-\mu}+C_{\varepsilon}(e^{\pi t^{2}}-1)|t|^{q-1}, \quad \mbox{for all} \hspace{0,2cm} t\in\mathbb{R},
\end{equation} 
which implies that
\begin{equation}\label{growth2}
F(t)\leq \tilde{\varepsilon}|t|^{2-\mu}+\tilde{C_{\varepsilon}}(e^{\pi t^{2}}-1)|t|^{q}, \quad \mbox{for all} \hspace{0,2cm} t\in\mathbb{R}.
\end{equation}
In view of the above estimates jointly with Hardy-Littlewood-Sobolev inequality, $I$ is well defined in $H^{1/2}(\mathbb{R})$. Furthermore, $I\in C^{1}(H^{1/2}(\mathbb{R}),\mathbb{R})$ and
\[
I^{\prime}(u)v={\frac{1}{2\pi}}\int_{\mathbb{R}}\int_{\mathbb{R}}\frac{[u(x)-u(y)][v(x)-v(y)]}{\vert x-y \vert^2}\,\mathrm{d}x\mathrm{d}y+\int_{\mathbb{R}}uv\,\mathrm{d}x-\int_{\mathbb{R}}\left(I_{\mu}\ast F(u)\right)f(u)v\,\mathrm{d}x.
\]
Thus, critical points of $I$ are weak solutions of Problem \eqref{problem0} and conversely.

Now, we prove that the energy functional defined in \eqref{ef} satisfies the Mountain Pass Geometry.

\begin{lemma}\label{01}
Suppose that \ref{f1} and \ref{f3} are satisfied. Then, the following conclusions hold:
\begin{itemize}
\item[(i)] There exist $\tau>0$ and $\varrho>0$ such that $I(u)\geq\tau$, provided that $\|u\|_{1/2}=\varrho$.
\item[(ii)] There exists $v\in H^{1/2}(\mathbb{R})$ with $\|v\|_{1/2}>\varrho$ such that $I(v)<0$.
\end{itemize}
\end{lemma}

\begin{proof}
Let us prove $(i)$. In view of \eqref{growth2} we get
\begin{equation}\label{rj1}
\|F(u)\|_{L^\frac{2}{2-\mu}(\mathbb{R})}\leq \varepsilon\|u\|_{L^2(\mathbb{R})}^{2-\mu}+C\left\{\int_{\mathbb{R}}\left[\left(e^{\pi u^{2}}-1\right)|u|^{q}\right]^{\frac{2}{2-\mu}}\,\mathrm{d}x\right\}^{\frac{2-\mu}{2}}.
\end{equation}
Consider $\varrho>0$ and suppose $\|u\|_{1/2}\leq \varrho$. By H\"older inequality we obtain
\begin{equation}\label{rj2}
\int_{\mathbb{R}}\left[\left(e^{\pi u^{2}}-1\right)|u|^{q}\right]^{\frac{2}{2-\mu}}\,\mathrm{d}x\leq \left[\int_{\mathbb{R}}\left(e^{\frac{4\pi \|u\|^{2}}{2-\mu}\left(\frac{u}{\|u\|_{1/2}}\right)^{2}}-1 \right)\,\mathrm{d}x \right]^{\frac{1}{2}}\|u\|_{L^\frac{4q}{2-\mu}(\mathbb{R})}^{\frac{2q}{2-\mu}}.
\end{equation}
If $\varrho\leq \sqrt{(2-\mu)}/2$, then we are able to apply Trudinger-Moser inequality (see Theorem A). Thus, \eqref{rj1}, \eqref{rj2} jointly with Sobolev embedding imply that
\[
\|F(u)\|_{L^\frac{2}{2-\mu}(\mathbb{R})}\leq \varepsilon\|u\|_{1/2}^{{2-\mu}}+\tilde{C}\|u\|_{1/2}^{{q}}.
\]
Hence, it follows from Hardy-Littlewood-Sobolev inequality that
\[
\int_{\mathbb{R}}(I_{\mu}\ast F(u))F(u)\,\mathrm{d}x\leq \varepsilon^2\|u\|_{1/2}^{2(2-\mu)}+\tilde{C}\|u\|_{1/2}^{{2q}}.
\]
Thus, we have
\[
I(u)\geq \frac{1}{2}\|u\|_{1/2}^{2}-\varepsilon^2\|u\|_{1/2}^{2(2-\mu)}-\tilde{C}\|u\|_{1/2}^{{2q}}.
\]
Since $2(2-\mu)>2$ and $2q>2$, there exist $\tau,\rho>0$  such that  if $\|u\|_{1/2}= \rho$, then $I(u)\geq \tau>0$.

Now in order to prove $(ii)$, take $u_0\in H^{1/2}(\mathbb{R})\setminus\{0\}$, $u_{0}\geq0$, $u_{0}\not\equiv0$ and set
\[w(t)= \frac{1}{2}\int_{\mathbb{R}} \left(I_{\mu}\ast F\left(t\dfrac{u_0}{\|u_0\|_{1/2}}\right)\right)F\left(t\dfrac{u_0}{\|u_0\|_{1/2}}\right)\,\mathrm{d}x, \quad \mbox{for } t>0.\]
It follows from \ref{f4} that
\[
\dfrac{w'(t)}{w(t)}> \dfrac{2K}{t}, \quad \mbox{for } t>0.
\]
Thus, integrating this over $[1, s\|u_0\|_{1/2}]$ with $s>1/\|u_0\|_{1/2}$, we can conclude that
\[
\frac{1}{2} (I_{\mu}\ast F({su_0}))F({su_0})\, \mathrm{d} x \geq\left(\frac{1}{2}\int_{\mathbb{R}} \left(I_{\mu}\ast F\left(\dfrac{u_0}{\|u_0\|_{1/2}}\right)\right)F\left(\dfrac{u_0}{\|u_0\|_{1/2}}\right)\, \mathrm{d} x\right)\|u_0\|_{1/2}^{2K}s^{2K}.
\]
Therefore, from \eqref{ef}, we get
\[I(su_0)< Cs^2 - Cs^{2K}, \quad \mbox{for } s> \dfrac{1}{\|u_0\|_{1/2}}.\]
Since $K>1$, taking $v=su_0$ with $s$ large enough, we have $(ii)$.
\end{proof}

In view of the preceding Lemma \ref{01}, we may apply Mountain Pass Theorem to get a (PS) sequence, i.e., $(u_{n})\subset H^{1/2}(\mathbb{R})$ such that
\[
I(u_{n})\rightarrow c:=\inf_{\gamma\in\Gamma}\max_{t\in[0,1]}I(\gamma(t)) \quad \mbox{and} \quad I^{\prime}(u_{n})\rightarrow0,
\]
where
\[
\Gamma:=\left\{\gamma \in C^{1}([0,1],H^{1/2}(\mathbb{R})):\gamma(0)=0, \ I(\gamma(1))<0\right\}.
\]

\section{Minimax Estimates}\label{mini}

The main difficulty in our work is the lack of compactness typical for elliptic problems in unbounded domains with nonlinearities with critical growth. In order to overcome this, we will make use of assumption \ref{f5} to control the minimax level in a suitable range where we are able to recover some compactness. For this purpose, let us consider the following sequence of nonnegative functions supported in $B_1$ given by
\[
u_n=\left\{
\begin{array}{cl}
\vspace{0,2cm}(\ln{n})^{1/2},&  \mbox{if } |x|< \frac{1}{n},\\
\vspace{0,2cm}\dfrac{\ln{\frac{1}{|x|}}}{(\ln{n})^{1/2}},& \mbox{if } \frac{1}{n}\leq |x|\leq 1,\\
0,& \mbox{if } |x|\geq 1.
\end{array}
\right.
\]
As pointed out in \cite{taka}, $u_{n}\in H^{1/2}(\mathbb{R})$ and we have
\[
\|(-\Delta)^{1/4} u_{n}\|^2_{L^2(\mathbb{R})}\leq\pi\left(1 +\frac{1}{C\ln(n)}\right):= \widetilde{C}_n. \]
 Thus by \eqref{normas}, for $n$ large enough, we have
\[
\|u_{n}\|_{1/2}^{2}\leq \widetilde{C}_n +2\left[\int_{-\frac{1}{n}}^{\frac{1}{n}}\ln(n)\,\mathrm{d}x+\frac{1}{\ln(n)}\int_{-1}^{-\frac{1}{n}}(\ln|x|)^{2}\,\mathrm{d}x+ \frac{1}{\ln(n)}\int_{\frac{1}{n}}^{1}(\ln|x|)^{2}\,\mathrm{d}x\right],
\]
which implies that $\|u_{n}\|_{1/2}^2\leq\widetilde{C}_n+\delta_{n}$ where
\begin{equation*}
\delta_{n}:=4 \left(\frac{n-1-\ln(n)}{n\ln(n)} \right).
\end{equation*}
Notice that
\begin{equation} \label{18}
\delta_n \rightarrow 0 \quad\mbox{ and } \quad \ln(n)\delta_n \rightarrow  4 \quad \mbox{as} \hspace{0,2cm} n\rightarrow+\infty.
\end{equation}
By setting $w_n:=\frac{u_{n}}{\sqrt{\widetilde{C}_n+\delta_{n}}}$ we obtain $\|w_n\|_{1/2}\leq 1$.
 
\begin{proposition}\label{mv}
Assume that $f$ satisfies \ref{f1}--\ref{f5}. Then
\[
c < \frac{2-\mu}{4}.
\]
\end{proposition}

\begin{proof}
 
Since
\[
c \leq  \max_{t\in [0,1]} I(tw) \leq \max_{t\geq 0} I(tw),
\]
it is sufficient to prove that there exists a function $w\in H^{1/2}(\mathbb{R})$, $\Vert w \Vert_{1/2}\leq 1$, such that 
\[
\max_{t\geq 0} I(tw)<\frac{2-\mu}{4}.
\]
In order to prove that, we claim that there exists $n_{0}$ such that
\[
\max_{t\geq 0} I(tw_{n_{0}})<\frac{2-\mu}{4}.
\]
Arguing by contradiction, we suppose that for all $n$ we have
\[
I(t_nw_n)=\max_{t\geq 0}  I(tw_n) \geq \frac{2-\mu}{4}.
\]
Since
\[
I(t_nw_n)=\frac{1}{2}\Vert t_nw_n \Vert_{1/2}^{2}-\frac{1}{2}\int_\mathbb{R}\left( \frac{1}{|x|^\mu}\ast F(t_nw_n) \right)F(t_nw_n)\,\mathrm{d}x,
\]
and $f$ is nonnegative, we obtain
\begin{equation}\label{17}
t_{n}^{2}\geq \frac{2-\mu}{2}.
\end{equation}
Moreover, as  $t_n$ satisfies
\[
\left.\frac{d}{dt}I(tw_n)\right|_{t=t_n}=0,
\]
it follows that
\begin{equation}\label{15}
t_{n}^{2}{\geq}\int_{\mathbb{R}}\left( \frac{1}{|x|^\mu}\ast F(t_nw_n) \right)t_nw_nf(t_nw_n)\,\mathrm{d}x.    
\end{equation}
On the other hand, from \eqref{HUB} we obtain
\begin{equation}\label{14}
tf(t)F(t)\geq (\beta_{0}-\varepsilon)C_0e^{2\pi t^2}t^{a+1}, \quad \text{for all }t\geq t_0.    
\end{equation}
Thus, for $n\in\mathbb{N}$ large enough, by using \eqref{15} and \eqref{14}, we have
\[
\begin{alignedat}{2}
t_{n}^{2}&\geq \int_{-\frac{1}{n}}^{\frac{1}{n}}\frac{t_n\sqrt{\ln n}}{\sqrt{\widetilde{C}_n+\delta_{n}}}f\left(\frac{ t_n\sqrt{\ln n}}{\sqrt{\widetilde{C}_n+\delta_{n}}}\right)\,\mathrm{d}y\int_{-\frac{1}{n}}^{\frac{1}{n}}\frac{F\left(\frac{ t_n\sqrt{\ln n}}{\sqrt{\widetilde{C}_n+\delta_{n}}}\right)}{|x-y|^\mu}\,\mathrm{d}x\\
&\geq (\beta_{0}-\varepsilon)C_0\exp\left(\frac{2\pi t_n^2\ln n}{\widetilde{C}_n+\delta_{n}
}\right)\left(\frac{ t_n\sqrt{\ln n}}{\sqrt{\widetilde{C}_n+\delta_{n}}}\right)^{a+1}\int_{-\frac{1}{n}}^{\frac{1}{n}}\int_{-\frac{1}{n}}^{\frac{1}{n}}\frac{\mathrm{d}x\,\mathrm{d}y}{|x-y|^\mu}.
\end{alignedat}
\]
Consequently,
\begin{equation}\label{novaUB}\begin{array}{rcl}
t_{n}^{2}&\geq &\displaystyle \frac{(\beta_{0}-\varepsilon) C_02^{2-\mu}}{(1-\mu)(2-\mu)}  \exp\left( \left[\frac{2\pi t_{n}^{2}}{\widetilde{C}_n+\delta_{n}} - (2-\mu)\right]\ln n \right) \left(\frac{ t_n\sqrt{\ln n}}{{\sqrt{\widetilde{C}_n+\delta_{n}}}}\right)^{a+1}\\
&\geq &	\displaystyle\frac{(\beta_{0}-\varepsilon) C_02^{2-\mu}}{(1-\mu)(2-\mu)}  \exp\left( \left[\frac{2\pi t_{n}^{2}}{\widetilde{C}_n+\delta_{n}} - (2-\mu)\right]\ln n \right).\end{array}
\end{equation}
Thus we conclude that $t_n^2$ is bounded. Moreover, it follows from \eqref{18} and definitions of $\widetilde{C}_n$ and $\delta_n$ that
\[
t_n^2\rightarrow {\frac{2-\mu}{2}}.
\]
We can rewrite \eqref{novaUB} as 
\[
t^2_n\geq\frac{(\beta_{0}-\varepsilon) C_02^{2-\mu}}{(1-\mu)(2-\mu)} \left(\frac{ t_n}{\sqrt{\widetilde{C}_n+\delta_{n}}}\right)^{a+1} \exp\left( \left[\frac{2\pi t_{n}^{2}}{ \widetilde{C}_n+\delta_{n}} - (2-\mu)\right]\ln n +\frac{(a+1)}{2}\ln(\ln(n))\right). 
\]
Since $t_n^2$ is bounded, there exists $C_1>0$ such that
\[ C_1\geq \left[\frac{2\pi t_{n}^{2}}{ \widetilde{C}_n+\delta_{n}} - (2-\mu)\right]\ln n +\frac{a+1}{2}\ln(\ln(n)).\]
Using \eqref{17}, we have
\[
\begin{alignedat}{2}
C_1 &\geq \left[\frac{\pi}{ \widetilde{C}_n+\delta_{n}}-1\right](2-\mu) \ln n +\frac{a+1}{2}\ln(\ln(n)).
\end{alignedat}
\]
Note that by \eqref{18} and the definition of $\widetilde{C}_n$
\[
\left[\frac{\pi}{ \widetilde{C}_n+\delta_{n}}-1\right](2-\mu) \ln n\rightarrow \frac{-\pi-C}{C\pi}(2-\mu), \quad \mbox{as } n\rightarrow\infty,
\]
for some $C>0$. Therefore, as $\ln(\ln(n))\rightarrow +\infty$, we have a contradiction.
\end{proof}

\section{Proof of Theorem~\ref{paper4A}}\label{fs}

Let $(u_{n})$ be the (PS) sequence obtained in Section~\ref{tpmm}. Thus, we have
\begin{equation}\label{ps}
\frac{1}{2}\|u_{n}\|^{2}_{1/2}-\frac{1}{2}\int_{\mathbb{R}}(I_{\mu}\ast F(u_{n}))F(u_{n})\,\mathrm{d}x\rightarrow c, \quad \mbox{as } n\rightarrow \infty,
\end{equation}
and
\begin{equation}\label{epsilon}
\left|\frac{1}{2\pi}\int_{\mathbb{R}}\int_{\mathbb{R}}\frac{[u_{n}(x)-u_{n}(y)][v(x)-v(y)]}{|x-y|^{2}}\,\mathrm{d}x\mathrm{d}y+\int_{\mathbb{R}}u_{n}v\,\mathrm{d}x-\int_{\mathbb{R}}(I_{\mu}\ast F(u_{n}))f(u_{n})v\,\mathrm{d}x \right|\leq \epsilon_{n}\|v\|_{1/2},
\end{equation}
for all $v\in H^{1/2}(\mathbb{R})$, where $\epsilon_{n}\rightarrow0$ as $n\rightarrow\infty$. Similarly to \cite[Lemma 2.4]{alvesetal} we conclude that $(u_n)$ is bounded in $H^{1/2}(\mathbb{R})$, up to a subsequence, $u_{n} \rightharpoonup u$ weakly in $H^{1/2}(\mathbb{R})$ and there holds
\[
\left[I_{\mu}\ast F(u_{n}) \right]F(u_{n})\rightarrow \left[I_{\mu}\ast F(u) \right]F(u), \quad \mbox{in} \hspace{0,2cm} L^{1}_{\mathrm{loc}}(\mathbb{R}).
\]
Let $u_{n}=u_{n}^{+}-u_{n}^{-}$, where $u_{n}^{+}(x)=\max\{u_{n}(x),0\}$ and $u_{n}^{-}=-\min\{u_{n}(x),0\}$. Since $f(t)=0$ for all $t\leq0$, by taking $v_{n}=-u_{n}^{-}$ and using the fact that $(u_{n})$ is a (PS) sequence, we obtain
\begin{align*}
o_{n}(1) = & I^{\prime}(u_{n})(-u_{n}^{-})\\
= & -\frac{1}{2\pi}\int_{\mathbb{R}}\int_{\mathbb{R}}\frac{[u_{n}(x)-u_{n}(y)][u_{n}^{-}(x)-u_{n}^{-}(y)]}{|x-y|^{2}}\,\mathrm{d}x\mathrm{d}y-\int_{\mathbb{R}}u_{n}u_{n}^{-}\,\mathrm{d}x\\
\geq & \Vert u_{n}^{-}\Vert_{1/2}, 
\end{align*}
where we have used that $u_{n}^{+},u_{n}^{-}\geq0$. Thus, $\Vert u_{n}^{-}\Vert_{1/2}\rightarrow0$, as $n\rightarrow\infty$. Hence, we have that
\[
\int_{\mathbb{R}}\int_{\mathbb{R}}\frac{[u_{n}^{+}(x)-u_{n}^{+}(y)][u_{n}^{-}(x)-u_{n}^{-}(y)]}{|x-y|^{2}}\,\mathrm{d}x\mathrm{d}y\rightarrow0, \quad \mbox{as } n\rightarrow\infty,
\]
which implies that $\Vert u_{n}\Vert_{1/2}=\Vert u_{n}^{+}\Vert_{1/2}+o_{n}(1)$. Therefore, $(u_{n}^{+})$ is also a (PS) sequence for funcional $I$. For this reason, we may suppose, without loss of generality, that $(u_{n})$ is a nonnegative Palais--Smale sequence.

Let us now prove that the weak limit $u$ yields actually a weak solution to Problem \eqref{problem0}. Following \cite[Lemma 2.4]{alvesetal}, let $\phi\in C_{0}^{\infty}(\mathbb{R})$ be such that $\mathrm{supp}\,\phi\subset \Omega^\prime$ satisfying $0\leq \phi\leq 1$ and $\phi\equiv 1$ in $\Omega\subset \Omega^\prime$ and define $v_n=\phi/(1+u_n)$. In view of Young's inequality one has
\begin{equation}\label{21}
\begin{alignedat}{2}
\int_{\mathbb{R}}\int_{\mathbb{R}}& \frac{[u_n(x)-u_n(y)][v_n(x)-v_n(y)]}{\vert x-y \vert^2}\,\mathrm{d}x\mathrm{d}y\leq \frac{1}{2}[u_n]_{1/2}^2+\frac{1}{2}[v_n]_{1/2}^2\\
&= \frac{1}{2}[u_n]_{1/2}^2+\frac{1}{2}\int_{\mathbb{R}}\int_{\mathbb{R}}\frac{[(1+u_n(y))\phi(x)-(1+u_n(x))\phi(y)]^2}{(1+u_n(x))^2(1+u_n(y))^2\vert x-y \vert^2}\,\mathrm{d}x\mathrm{d}y\\
&\leq \frac{1}{2}[u_n]_{1/2}^2+\frac{1}{2}\int_{\mathbb{R}}\int_{\mathbb{R}}\frac{[(1+u_n(y))\phi(x)-(1+u_n(x))\phi(y)]^2}{\vert x-y \vert^2}\,\mathrm{d}x\mathrm{d}y\\
&\leq C\left( [u_n]_{1/2}^2 + [\phi]_{1/2}^2 + [u_n\phi]_{1/2}^2\right).
\end{alignedat}
\end{equation}
Notice that
\begin{equation}\label{est1}
\begin{alignedat}{2}
\int_{\mathbb{R}}\int_{\mathbb{R}}&\frac{|u_{n}(x)\phi(x)-u_{n}(y)\phi(y)|^{2}}{|x-y|^{2}}\,\mathrm{d}x\mathrm{d}y\\
&\leq C_{1}\int_{\mathbb{R}}\int_{\mathbb{R}}\left[ \frac{|u_{n}(x)\phi(x)-u_{n}(y)\phi(x)|^{2}}{|x-y|^{2}}+\frac{|u_{n}(x)|^{2}|\phi(x)-\phi(y)|^{2}}{|x-y|^{2}}\right]\,\mathrm{d}x\mathrm{d}y\\
& \leq \tilde{C_{1}}[u_{n}]_{1/2}^{2}+C_{2}(\phi)\int_{\mathbb{R}}u_{n}^{2}\,\mathrm{d}x\\
&\leq C(\phi)\Vert u_n \Vert_{1/2}^2,
\end{alignedat}
\end{equation}
where $C_{2}(\phi)$ is a constant which depends on $\phi$. By using \eqref{epsilon}, \eqref{21} and \eqref{est1} we obtain
\[
\begin{alignedat}{2}
\int_{\Omega}&\left[ \frac{1}{\vert x \vert^\mu}\ast F(u_n)\right]\frac{f(u_n)}{1+u_n}\,\mathrm{d}x\leq \int_{\mathbb{R}}\left[ \frac{1}{\vert x \vert^\mu}\ast F(u_n)\right]\frac{f(u_n)\phi}{1+u_n}\,\mathrm{d}x\\
&=\frac{1}{2\pi}\int_{\mathbb{R}}\int_{\mathbb{R}}\frac{[u_n(x)-u_n(y)][v_n(x)-v_n(y)]}{\vert x-y \vert^2}\,\mathrm{d}x\mathrm{d}y+\int_{\mathbb{R}}u_nv_n\,\mathrm{d}x+\epsilon_n\Vert v_n \Vert_{1/2}\\
&\leq C\left( [u_n]_{1/2}^2 + [\phi]_{1/2}^2 + [u_n\phi]_{1/2}^2\right)+\int_{\Omega^\prime}u_n\,\mathrm{d}x+\epsilon_n\Vert v_n \Vert_{1/2}\\
&\leq \tilde{C}(\phi)\Vert u_n \Vert_{1/2}^2 + C_3(\phi)+\int_{\Omega^\prime}u_n\,\mathrm{d}x+\epsilon_n\Vert v_n \Vert_{1/2}.
\end{alignedat}
\]
Since $(u_n)$ is bounded in $H^{1/2}(\mathbb{R})$ and $u_n\rightarrow u$ in $L^1(\Omega^\prime)$ we conclude that
\[
\int_{\Omega}\left[ \frac{1}{\vert x \vert^\mu}\ast F(u_n)\right]\frac{f(u_n)}{1+u_n}\,\mathrm{d}x\leq \overline{C}(\phi).
\]
Thus, by a Radon-Nikodym argument, we can conclude that
\[
\lim_{n\rightarrow\infty}\int_{\mathbb{R}}(I_{\mu}\ast F(u_{n}))f(u_{n})\phi\,\mathrm{d}x=\int_{\mathbb{R}}(I_{\mu}\ast F(u))f(u)\phi\,\mathrm{d}x, \quad \mbox{for all } \phi\in C^{\infty}_{c}(\mathbb{R}).
\]
Therefore, $u$ is a weak solution for Problem \eqref{problem0}.  If $u\neq0$, then the proof is done. Suppose that $u=0$. We claim that there exists $R,\delta>0$ and a sequence $(y_{n})\subset\mathbb{Z}$ such that
\begin{equation}\label{nonvanish}
\lim_{n\rightarrow+\infty}\int_{y_{n}-R}^{y_{n}+R}|u_{n}|^{2}\,\mathrm{d}x\geq\delta.
\end{equation}
Suppose by contradiction that \eqref{nonvanish} does not hold. Thus, for any $R>0$, there holds
\[
\lim_{n\rightarrow+\infty}\sup_{y\in\mathbb{R}}\int_{y-R}^{y+R}|u_{n}|^{2}\,\mathrm{d}x=0.
\]
In view of Lemma~\ref{lion}, $u_{n}\rightarrow0$ strongly in $L^{p}(\mathbb{R})$, for $2<p<\infty$. Similarly to \cite{alvesetal} we may conclude that
\begin{equation}\label{erj11}
\left[I_{\mu}\ast F(u_{n}) \right]F(u_{n})\rightarrow0, \quad \mbox{in } L^{1}(\mathbb{R}).
\end{equation}
Hence, in view of Proposition \ref{mv}, \eqref{ps} and \eqref{erj11} one has 
\[
\lim_{n\rightarrow+\infty}\|u_{n}\|_{1/2}^{2}=2c<\frac{2-\mu}{2}.
\]
Thus, there exists $\delta>0$ small and $n_{0}\in\mathbb{N}$ large such that
\begin{equation}\label{erj13}
\|u_{n}\|_{1/2}^{2}\leq \frac{2-\mu}{2}(1-\delta), \quad \mbox{for all } n\geq n_{0}.
\end{equation}
In light of Hardy-Littlewood-Sobolev inequality we have
\[
\int_{\mathbb{R}}\left(I_{\mu}\ast F(u_n) \right)f(u_n)u_n\,\mathrm{d}x\leq C\|F(u_{n})\|_{L^\frac{2}{2-\mu}(\mathbb{R})}\|f(u_{n})u_{n} \|_{L^\frac{2}{2-\mu}(\mathbb{R})}.
\]
By using \eqref{growth}, for any $\varepsilon>0$ and $q>2$ there is $\delta>0$ such that
\[
\|f(u_{n})u_{n} \|_{L^\frac{2}{2-\mu}(\mathbb{R})} \leq \varepsilon\|u_{n}\|_{L^2(\mathbb{R})}^{2-\mu}+C_{\varepsilon}\left[\int_{\mathbb{R}}(e^{\pi u_{n}^{2}}-1)^{\frac{2}{2-\mu}}|u_{n}|^{\frac{2q}{2-\mu}}\,\mathrm{d}x \right]^{\frac{2-\mu}{2}}.
\]
Let us consider $\sigma>1$ close to $1$ and $r,r'>1$ such that $1/r+1/r'=1$. Thus, one has
\[
\left[\int_{\mathbb{R}}(e^{\pi u_{n}^{2}}-1)^{\frac{2}{2-\mu}}|u_{n}|^{\frac{2q}{2-\mu}}\,\mathrm{d}x \right]^{\frac{2-\mu}{2}}\leq \|u_{n}\|_{L^\frac{2qr'}{2-\mu}(\mathbb{R})}^{q}\left[\int_{\mathbb{R}}(e^{\frac{2\sigma r}{2-\mu}\pi u_{n}^{2}}-1)\,\mathrm{d}x \right]^{\frac{2-\mu}{2r}}. 
\]
By choosing $\sigma,r>1$ sufficiently close to $1$ such that
\[
1<\sigma r<\frac{1}{1-\delta} \quad \mbox{and} \quad \frac{2qr'}{2-\mu}>2,
\]
it follows from \eqref{erj13} that
\[
\frac{2\sigma r}{2-\mu}\|u_{n}\|_{1/2}^{2}<1, \quad \mbox{for all } n\geq n_{0}.
\]
Thus, in view of Theorem \ref{tmm} we obtain
\begin{equation}\label{erj14}
\int_{\mathbb{R}}(e^{\frac{2\sigma r}{2-\mu}\pi u_{n}^{2}}-1)\,\mathrm{d}x=\int_{\mathbb{R}}(e^{\frac{2\sigma r}{2-\mu}\|u_{n}\|^{2}\pi \frac{u_{n}^{2}}{\|u_{n}\|^{2}_{1/2}}}-1)\,\mathrm{d}x\leq C, \quad \mbox{for all } n\geq n_{0}.
\end{equation}
Therefore, by using Lemma~\ref{lion} and combining \eqref{erj13}--\eqref{erj14} we conclude that
\[
\int_{\mathbb{R}}\left(I_{\mu}\ast F(u_n) \right)f(u_n)u_n\,\mathrm{d}x\rightarrow0, \quad \mbox{as } n\rightarrow\infty.
\]
Since $(u_{n})$ is a (PS) sequence we have that
\[
0<c=\frac{1}{2}\|u_{n}\|_{1/2}^{2}+o(1) \quad \mbox{and} \quad o(1)=\|u_{n}\|_{1/2}^{2}, 
\]
which is not possible. Therefore, \eqref{nonvanish} is satisfied. Since $I$ does not depend on $x$, we can say that it is a periodic functional with respect to this variable. In view of this periodicity of the energy functional  we are able to use a standard argument to get a (PS) sequence which for simplicity we also denote $(u_{n})$, such that $u_{n}\rightharpoonup u\neq0$ and $I^{\prime}(u)=0$, that is, $u$ is a nontrivial weak solution to Problem \eqref{problem0} (see \cite{ManassesYane}), which finishes the proof of Theorem~\ref{paper4A}.

\begin{remark} 
Let $u\in H^{1/2}(\mathbb{R})$ be the weak solution obtained in Theorem~\ref{paper4A}. By choosing the negative part $u^{-}\in H^{1/2}(\mathbb{R})$ as test function and using the inequality
\[
[u(x)-u(y)][u^{-}(x)-u^{-}(y)]\geq \vert u^{-}(x)-u^{-}(y)\vert^{2}, \quad \mbox{for all } x,y\in\mathbb{R},
\]
one may conclude that $\Vert u^{-}\Vert_{1/2}\leq 0$. Therefore, the weak solution $u$ is nonnegative. By using regularity theory and \cite[Theorem 1.2]{pm} one may conclude that $u$ is positive.
\end{remark}


\begin{thebibliography}{99}

\bibitem{ad}
Adimurthi, Yadava, S.L.:
Multiplicity results for semilinear elliptic equations in bounded domain of $\mathbb{R}^2$ involving critical exponent,
Ann. Sc. Norm. Super. Pisa \textbf{17}, 481–504 (1990).

\bibitem{Ub}
Albuquerque, F. S. B., Ferreira, M. C., Severo, U. B.:
Ground state solutions for a nonlocal equation in $\mathbb{R}^{2}$ involving vanishing potentials and exponential critical growth,
arXiv.

\bibitem{alvesetal}
Alves, C. O., Cassani, D., Tarsi, C., Yang, M.:
Existence and concentration of ground state solutions for a critical nonlocal Schr\"{o}dinger equation in $\mathbb{R}^{2}$,
J. Differential Equations \textbf{261}, no. 3, 1933--1972 (2016). 

\bibitem{alvmimbo}
Alves, C. O., Yang, M.:
Existence of solutions for a nonlocal variational problem in $\mathbb{R}^2$ with exponential critical growth,
J. Convex Anal. \textbf{24}, 1197-1215 (2017). 

\bibitem{nlaplacian}
Arora, R., Giacomoni, J., Mukherjee, T., Sreenadh, K.:
$n$-Kirchhoff-Choquard equations with exponential nonlinearity,
Nonlinear Anal. \textbf{186}, 113–144 (2019).

\bibitem{caffa}
Caffarelli, L.:
Non-local diffusions, drifts and games,
Nonlinear Partial Differential Equations, Abel Symp. \textbf{7}, Springer, Heidelberg, 37--52 (2012).

\bibitem{cao}
Cao, D.
Nontrivial solution of semilinear elliptic equation with critical exponent in $\mathbb{R}^2$,
 Comm. Partial Differential Equations 17 (1992) 407-435.

\bibitem{choq1}
Choquard, P., Stubbe, J.:
The one-dimensional Schr\"odinger-Newton equations,
Lett. Math. Phys. \textbf{81}, 177–184 (2007).

\bibitem{choq2}
Choquard, P., Stubbe, J., Vuffray, M.:
Stationary solutions of the Schr\"odinger-Newton model - an ODE approach,
Differential Integral Equations 21, 665–679 (2008).

\bibitem{Miyagaki}
 de Figueiredo, D. G., Miyagaki, O. H., Ruf, B.:
Elliptic equations in $\mathbb{R}^{2}$ with nonlinearities in the critical growth range,
Calc. Var. Partial Differential Equations 3, 139--153 (1995).

\bibitem{pm}
Del Pezzo, L.M., Quaas, A.:
A Hopf's lemma and a strong minimum principle for the fractional p-Laplacian,
J. Differential Equations \textbf{263}, no. 1, 765--778 (2017).

\bibitem{ManassesYane}
de Souza, M., Araújo, Y.L.:
On nonlinear perturbations of a periodic fractional Schr\"odinger equation with critical exponential growth,
Math. Nachr. \textbf{289}, 610–625 (2016).

\bibitem{guia}
Nezza, E. Di, Palatucci, G., Valdinoci, E.:
Hitchhiker's guide to the fractional Sobolev spaces,
Bull. Sci. Math. \textbf{136}, 521-573 (2012).

\bibitem{Pawan}
Giacomoni, J., Mishra, P, Sreenadh, K.:
Critical growth problems for $1/2$-Laplacian in $\mathbb{R}$,
Differ. Equ. Appl. \textbf{8}, 295-317 (2016).

\bibitem{iulamar}
Iula, S., Maalaoui, A., Martinazzi, L.:
A fractional Moser-Trudinger type inequality in one dimension and its critical points,
Differential Integral Equations \textbf{29}, no. 5/6, 455--492 (2016).

\bibitem{Wadade}
 Kozono, H., Sato, T., Wadade, H.: 
Upper bound of the best constant of a Trudinger-Moser inequality and its application to a Gagliardo-Nirenberg inequality,
Indiana Univ. Math. J. \textbf{55}, 1951--1974 (2006).

\bibitem{choquard}
Lieb, E.:
Existence and uniqueness of the minimizing solution of Choquard nonlinear equation,
Studies in Appl. Math. \textbf{57}, 93--105 (1976/77).

\bibitem{Li}
Li, S., Shen,  Z., Yang, M.:
Multiplicity of solutions for a nonlocal nonhomogeneous elliptic equation with critical exponential growth,
J. Math. Anal. Appl. \textbf{475}, 1685-1713 (2019).

\bibitem{lions}
Lions, P. L.:
The Choquard equation and related questions,
Nonlinear Anal. \textbf{4}, 1063--1072 (1980).

\bibitem{lionss} 
Lions, P.L.:
The concentration-compactness principle in the calculus of variations. The locally compact case,
Ann. Inst. H. Poincar\'e Anal. Non Lin\'{e}aire \textbf{1}, 109-145 and 223-283 (1984).

\bibitem{moroz}
Moroz, V., Van Schaftingen, J.:
Existence of ground states for a class of nonlinear Choquard equations,
Trans. Amer. Math. Soc. {\bf 367} , no. 9, 6557--6579 (2015).

\bibitem{moroz2}
Moroz, V., Van Schaftingen, J.:
Ground states of nonlinear Choquard equations: Existence, qualitative properties and decay asymptotics,
J. Funct. Anal. \textbf{265}, 153--184 (2013).

\bibitem{Ozawa}
 Ozawa, T.:
On critical cases of Sobolev's inequalities,
J. Funct. Anal. \textbf{127}, 259--269 (1995).

\bibitem{pekar}
Pekar, S.:
Untersuchung\"uber Die Elektronentheorie Der Kristalle,
Akademie Verlag, Berlin, 1954.

\bibitem{penrose}
Penrose, R.:
On gravity role in quantum state reduction,
Gen. Relativ. Gravitat. \textbf{28}, 581600 (1996).

\bibitem{subcritical}
Sun, X., Zhu, A.:
Multi-peak solutions for nonlinear Choquard equation in the plane,
J. Math. Phys. \textbf{60}, 1--18 (2019). 

\bibitem{taka}
Takahashi, F.:
Critical and subcritical fractional Trudinger-Moser-type inequalities on $\mathbb{R}$.
Advances in Nonlinear Analysis \textbf{1}, 868--884 (2019).

\bibitem{Mimbo}
 Yang, M.:
Semiclassical ground state solutions for a Choquard type equation in $\mathbb{R}^2$ with critical exponential growth,
ESAIM Control Optim. Calc. Var. \textbf{24}, 177-209 (2018).

\end{thebibliography}
\end{document}